\documentclass{amsart}

\newtheorem{theorem}{Theorem}[section]
\newtheorem{lemma}[theorem]{Lemma}
\newtheorem{corollary}[theorem]{Corollary}

\theoremstyle{definition}

\theoremstyle{remark}
\newtheorem{remark}[theorem]{Remark}

\usepackage{amsmath}
\usepackage{amssymb}
\usepackage{graphicx}
\usepackage{mathrsfs}
\usepackage{algorithm}
\usepackage{algorithmicx}
\usepackage{appendix}
\usepackage{enumitem}
\usepackage{mdframed}
\usepackage{booktabs}
\usepackage{array}
\usepackage{caption}
\usepackage[colorlinks, linkcolor=blue, anchorcolor=blue, citecolor=blue]{hyperref}
\AtBeginDocument{
                 \label{CorrectFirstPageLabel}
                 
                }

\DeclareMathOperator*{\rank}{rank}
\DeclareMathOperator*{\diag}{diag}

\DeclareMathOperator*{\Span}{span}

\numberwithin{equation}{section}

\begin{document}

\title[Two-level convergence identity]{A new characterization of the convergence factor of two-level methods}

\author{Xuefeng Xu}
\address{School of Mathematics, Southeast University, Nanjing 211189, China}
\email{xuxuefeng@lsec.cc.ac.cn; xuxuefeng@seu.edu.cn}






\begin{abstract}
Multilevel methods are among the most efficient numerical methods for solving large-scale systems of equations that arise from discretized partial differential equations. Two-level convergence theory plays a fundamental role in the analysis and design of multilevel methods. In this paper, we present a concise and easy-to-use identity for characterizing the convergence factor of two-level methods, whose hierarchical spaces can be either overlapping or non-overlapping. In order to illustrate its usability and convenience, we give several applications, which offer new insights into the design of multilevel methods.
\end{abstract}

\maketitle



\section{Introduction}

Multilevel methods are among the most efficient numerical techniques for solving large-scale linear systems arising from the discretization of partial differential equations; see, e.g.,~\cite{Hackbusch1985,Briggs2000,Trottenberg2001,Vassilevski2008}. The foundation of multilevel methods is a two-level scheme, which consists of two complementary error-reduction processes: \textit{smoothing} (or \textit{local relaxation}) and \textit{coarse-level correction}. Usually, the smoothing process is a simple iterative method (e.g., the weighted Jacobi and Gauss--Seidel iterations), which is only effective at reducing high-frequency (or oscillatory) error in general. The remaining low-frequency (or smooth) error will be further reduced on a coarser level by the correction process. Satisfactory convergence can be achieved when the error-reduction processes complement each other very well.

Assume that $A\in\mathbb{R}^{n\times n}$ is symmetric positive definite (SPD) and $\mathbf{f}\in\mathbb{R}^{n}$. Consider two-level methods for solving the linear system
\begin{equation}\label{system}
A\mathbf{u}=\mathbf{f}.
\end{equation}
Let $S\in\mathbb{R}^{n\times n_{\rm s}}\,(n_{\rm s}\leq n)$ be of full column rank, and let $M_{\rm s}\in\mathbb{R}^{n_{\rm s}\times n_{\rm s}}$ be a nonsingular matrix such that $M_{\rm s}+M_{\rm s}^{T}-S^{T}AS$ is positive definite (e.g., $M_{\rm s}$ is chosen as the lower triangular part of $S^{T}AS$). For a given initial guess $\mathbf{u}^{(0)}\in\mathbb{R}^{n}$, the smoothing process can be described as follows:
\begin{equation}\label{smoothing}
\mathbf{u}^{(k+1)}=\mathbf{u}^{(k)}+SM_{\rm s}^{-1}S^{T}\big(\mathbf{f}-A\mathbf{u}^{(k)}\big).
\end{equation}
Indeed, the smoothing iteration~\eqref{smoothing} covers both \textit{local} ($n_{\rm s}<n$) and \textit{global} ($n_{\rm s}=n$) cases. In the global case, $S$ will be simply chosen as the $n\times n$ identity matrix $I_{n}$. In general, the iteration~\eqref{smoothing} can only reduce high-frequency error effectively. The remaining low-frequency error will be treated by the correction process
\begin{equation}\label{correction}
\mathbf{u}^{(\ell+1)}=\mathbf{u}^{(\ell)}+PA_{\rm c}^{-1}P^{T}\big(\mathbf{f}-A\mathbf{u}^{(\ell)}\big),
\end{equation}
where $P\in\mathbb{R}^{n\times n_{\rm c}} \ (n_{\rm c}<n)$ is a prolongation (or interpolation) matrix of rank $n_{\rm c}$ and
\begin{equation}\label{Ac}
A_{\rm c}:=P^{T}AP
\end{equation}
is known as the \textit{Galerkin coarse-level matrix}. From~\eqref{correction}, we have
\begin{displaymath}
\mathbf{u}-\mathbf{u}^{(\ell+1)}=\big(I_{n}-PA_{\rm c}^{-1}P^{T}A\big)\big(\mathbf{u}-\mathbf{u}^{(\ell)}\big).
\end{displaymath}
Since $I_{n}-PA_{\rm c}^{-1}P^{T}A$ is an $A$-orthogonal projector along (or parallel to) the column space $\mathcal{R}(P)$ onto the null space $\mathcal{N}(P^{T}A)$, fast convergence will be achieved if $\mathcal{R}(P)$ can cover most of the low-frequency error.

A symmetric two-level scheme for solving~\eqref{system} is described by Algorithm~\ref{alg:TL}, in which the pre- and postsmoothing steps are performed symmetrically. In particular, if $n_{\rm s}=n$ and $S=I_{n}$, then Algorithm~\ref{alg:TL} will reduce to a \textit{two-grid} method, which is not our focus here, because the properties of two-grid and multigrid methods have been well studied in the literature; see, e.g.,~\cite{Falgout2004,Falgout2005,Zikatanov2008,MacLachlan2014,Notay2015,XZ2017,XXF2018,XXF2022-1,XXF2022-2,XXF2025}. Another special case is that $(S \,\ P)\in\mathbb{R}^{n\times(n_{\rm s}+n_{\rm c})}$ is square and nonsingular, in which case Algorithm~\ref{alg:TL} reduces to a \textit{two-level hierarchical basis} method~\cite{Braess1981,Axelsson1983,Yserentant1986,Bank1988}.

\begin{algorithm}[!htbp]

\caption{\ Two-level method.}\label{alg:TL}

\smallskip

\begin{algorithmic}[1]

\State Presmoothing: $\mathbf{u}^{(1)}\gets\mathbf{u}^{(0)}+SM_{\rm s}^{-1}S^{T}\big(\mathbf{f}-A\mathbf{u}^{(0)}\big)$

\smallskip

\State Restriction: $\mathbf{r}_{\rm c}\gets P^{T}\big(\mathbf{f}-A\mathbf{u}^{(1)}\big)$

\smallskip

\State Coarse-level correction: $\mathbf{e}_{\rm c}\gets A_{\rm c}^{-1}\mathbf{r}_{\rm c}$

\smallskip

\State Prolongation: $\mathbf{u}^{(2)}\gets\mathbf{u}^{(1)}+P\mathbf{e}_{\rm c}$

\smallskip

\State Postsmoothing: $\mathbf{u}_{\rm TL}\gets\mathbf{u}^{(2)}+SM_{\rm s}^{-T}S^{T}\big(\mathbf{f}-A\mathbf{u}^{(2)}\big)$

\smallskip

\end{algorithmic}

\end{algorithm}

It can be easily verified that a sufficient and necessary condition for the convergence factor of Algorithm~\ref{alg:TL} to be less than $1$ is
\begin{displaymath}
\rank(S \,\ P)=n,
\end{displaymath}
or, equivalently,
\begin{displaymath}
\mathcal{R}(S)+\mathcal{R}(P)=\mathbb{R}^{n},
\end{displaymath}
i.e., for any $\mathbf{v}\in\mathbb{R}^{n}$, there exist $\mathbf{v}_{\rm s}\in\mathbb{R}^{n_{\rm s}}$ and $\mathbf{v}_{\rm c}\in\mathbb{R}^{n_{\rm c}}$ such that $\mathbf{v}=S\mathbf{v}_{\rm s}+P\mathbf{v}_{\rm c}$. Under the above condition, an identity for characterizing the convergence factor of Algorithm~\ref{alg:TL} has been established in~\cite[Theorem~4.1]{Falgout2005} (see~\cite{XZ2002,Zikatanov2008} for an abstract version), despite it involves a `sup-inf' expression. In general, however, it is tough to get an optimal vector decomposition involved in the `sup-inf' expression. Motivated by this observation, we attempt to establish an easy-to-use convergence identity for Algorithm~\ref{alg:TL}.

In this paper, we present a concise characterization of the convergence factor of Algorithm~\ref{alg:TL}; see~\eqref{new-id}. In some cases, it is more convenient and tractable compared to the existing one. To illustrate this point, we provide several applications, which are briefly described as follows.
\begin{itemize}

\item Firstly, we rigorously prove that the two-level convergence factor decreases as the hierarchical space $\mathcal{R}(P)$ expands.

\item Secondly, we derive a class of optimal prolongation matrices for minimizing the two-level convergence factor. 

\item Finally, we develop a new convergence theory for a class of reduction-based two-level methods.

\end{itemize}

The rest of this paper is organized as follows. In Section~\ref{sec:pre}, we review the existing convergence identity for Algorithm~\ref{alg:TL} and introduce a class of reduction-based two-level methods. In Section~\ref{sec:main}, we present a concise characterization of the convergence factor of Algorithm~\ref{alg:TL}, followed by discussions on its applications. In Section~\ref{sec:con}, we make some concluding remarks.

\section{Preliminaries} \label{sec:pre}

For convenience, we start with some notation used in the subsequent discussions.

\begin{itemize}

\item[--] We write $S_{1}\succeq S_{2}$ (resp., $S_{1}\succ S_{2}$) if $S_{1}-S_{2}$ is positive semidefinite (resp., positive definite), provided that $S_{1}$ and $S_{2}$ are two real symmetric matrices of the same order.

\item[--] $I_{n}$ denotes the $n\times n$ identity matrix (or $I$ when its size is clear from context).

\item[--] $\mathcal{R}(\cdot)$ and $\mathcal{N}(\cdot)$ denote the column space (or range) and null space (or kernel) of a matrix, respectively.

\item[--] $\lambda_{\min}(\cdot)$ and $\lambda_{\max}(\cdot)$ denote the smallest and largest eigenvalues of a matrix, respectively.

\item[--] $\lambda_{i}(\cdot)$ denotes the $i$th smallest eigenvalue of a matrix.

\item[--] $\lambda(\cdot)$ denotes the spectrum of a matrix.

\item[--] $\|\cdot\|_{2}$ denotes the spectral norm of a matrix.

\item[--] $\|\cdot\|_{A}$ denotes the norm induced by SPD matrix $A\in\mathbb{R}^{n\times n}$: $\|\mathbf{v}\|_{A}=\sqrt{\mathbf{v}^{T}A\mathbf{v}}$ for any $\mathbf{v}\in\mathbb{R}^{n}$; $\|B\|_{A}=\max\limits_{\mathbf{v}\in\mathbb{R}^{n}\backslash\{0\}}\frac{\|B\mathbf{v}\|_{A}}{\|\mathbf{v}\|_{A}}$ for any $B\in\mathbb{R}^{n\times n}$.

\end{itemize}

\subsection{The existing convergence identity}

Several basic assumptions involved in the analysis of Algorithm~\ref{alg:TL} are summarized as follows.

\begin{itemize}

\item Let $S\in\mathbb{R}^{n\times n_{\rm s}} \ (n_{\rm s}\leq n)$ and $P\in\mathbb{R}^{n\times n_{\rm c}} \ (n_{\rm c}<n)$ be of full column rank. In particular, if $n_{\rm s}=n$, then $S$ will be simply chosen as $I_{n}$.

\item The columns of $(S \,\ P)\in\mathbb{R}^{n\times(n_{\rm s}+n_{\rm c})}$ can span the whole space $\mathbb{R}^{n}$, namely,
\begin{displaymath}
\mathcal{R}(S)+\mathcal{R}(P)=\mathbb{R}^{n},
\end{displaymath}
or, equivalently,
\begin{equation}\label{rank}
\rank(S \,\ P)=n.
\end{equation}

\item Let $M_{\rm s}\in\mathbb{R}^{n_{\rm s}\times n_{\rm s}}$ be a nonsingular matrix such that $M_{\rm s}+M_{\rm s}^{T}\succ A_{\rm s}$, where $A_{\rm s}:=S^{T}AS$.

\end{itemize}

Define
\begin{equation}\label{piA}
\Pi_{A}:=PA_{\rm c}^{-1}P^{T}A,
\end{equation}
where $A_{\rm c}$ is defined by~\eqref{Ac}. From Algorithm~\ref{alg:TL}, we have
\begin{equation}\label{propagation}
\mathbf{u}-\mathbf{u}_{\rm TL}=E_{\rm TL}\big(\mathbf{u}-\mathbf{u}^{(0)}\big),
\end{equation}
where $E_{\rm TL}$, called the \textit{iteration matrix} (or \textit{error propagation matrix}) of Algorithm~\ref{alg:TL}, is given by
\begin{equation}\label{ETL1}
E_{\rm TL}=\big(I-SM_{\rm s}^{-T}S^{T}A\big)(I-\Pi_{A})\big(I-SM_{\rm s}^{-1}S^{T}A\big).
\end{equation}
Define
\begin{displaymath}
\overline{M}_{\rm s}:=M_{\rm s}\big(M_{\rm s}+M_{\rm s}^{T}-A_{\rm s}\big)^{-1}M_{\rm s}^{T}.
\end{displaymath}
Then, $E_{\rm TL}$ can be expressed as
\begin{equation}\label{ETL2}
E_{\rm TL}=I-B_{\rm TL}^{-1}A,
\end{equation}
where
\begin{equation}\label{inv-BTL}
B_{\rm TL}^{-1}=S\overline{M}_{\rm s}^{-1}S^{T}+\big(I-SM_{\rm s}^{-T}S^{T}A\big)PA_{\rm c}^{-1}P^{T}\big(I-ASM_{\rm s}^{-1}S^{T}\big).
\end{equation}

\begin{remark}
We remark that the inverse $B_{\rm TL}^{-1}$ is well defined and $B_{\rm TL}^{-1}\succ 0$. In fact, we get from~\eqref{inv-BTL} that, for any $\mathbf{v}\in\mathbb{R}^{n}$,
\begin{displaymath}
\mathbf{v}^{T}B_{\rm TL}^{-1}\mathbf{v}=\underbrace{\mathbf{v}^{T}S\overline{M}_{\rm s}^{-1}S^{T}\mathbf{v}}_{\geq\,0}+\underbrace{\mathbf{v}^{T}\big(I-SM_{\rm s}^{-T}S^{T}A\big)PA_{\rm c}^{-1}P^{T}\big(I-ASM_{\rm s}^{-1}S^{T}\big)\mathbf{v}}_{\geq\,0}\geq 0,
\end{displaymath}
which yields $B_{\rm TL}^{-1}\succeq 0$. If $\mathbf{v}^{T}B_{\rm TL}^{-1}\mathbf{v}=0$, then
\begin{displaymath}
\mathbf{v}^{T}S\overline{M}_{\rm s}^{-1}S^{T}\mathbf{v}=0 \quad \text{and} \quad \mathbf{v}^{T}\big(I-SM_{\rm s}^{-T}S^{T}A\big)PA_{\rm c}^{-1}P^{T}\big(I-ASM_{\rm s}^{-1}S^{T}\big)\mathbf{v}=0,
\end{displaymath}
from which we deduce that $\mathbf{v}\in\mathcal{N}(S^{T})\cap\mathcal{N}(P^{T})$ and hence $\mathbf{v}\in\mathcal{N}\big((S \,\ P)^{T}\big)$. This, together with~\eqref{rank}, leads to $\mathbf{v}=0$. Thus, the positive definiteness of $B_{\rm TL}^{-1}$ is proved.
\end{remark}

According to~\eqref{propagation}, we deduce that
\begin{displaymath}
\|\mathbf{u}-\mathbf{u}_{\rm TL}\|_{A}\leq\|E_{\rm TL}\|_{A}\big\|\mathbf{u}-\mathbf{u}^{(0)}\big\|_{A}.
\end{displaymath}
In light of~\eqref{ETL1} and~\eqref{ETL2}, we have
\begin{displaymath}
I-A^{\frac{1}{2}}B_{\rm TL}^{-1}A^{\frac{1}{2}}=\big(I-A^{\frac{1}{2}}SM_{\rm s}^{-T}S^{T}A^{\frac{1}{2}}\big)\big(I-A^{\frac{1}{2}}\Pi_{A}A^{-\frac{1}{2}}\big)\big(I-A^{\frac{1}{2}}SM_{\rm s}^{-1}S^{T}A^{\frac{1}{2}}\big).
\end{displaymath}
Since $A^{\frac{1}{2}}\Pi_{A}A^{-\frac{1}{2}}$ is an $L^{2}$-orthogonal projector (so $I\succeq A^{\frac{1}{2}}\Pi_{A}A^{-\frac{1}{2}}$), it follows that
\begin{displaymath}
I-A^{\frac{1}{2}}B_{\rm TL}^{-1}A^{\frac{1}{2}}\succeq 0.
\end{displaymath}
Then
\begin{equation}\label{norm-ETL}
\|E_{\rm TL}\|_{A}=\big\|I-A^{\frac{1}{2}}B_{\rm TL}^{-1}A^{\frac{1}{2}}\big\|_{2}=\lambda_{\max}\big(I-A^{\frac{1}{2}}B_{\rm TL}^{-1}A^{\frac{1}{2}}\big)=1-\lambda_{\min}\big(B_{\rm TL}^{-1}A\big),
\end{equation}
which is referred to as the \textit{convergence factor} of Algorithm~\ref{alg:TL}.

The following theorem gives an identity for characterizing the convergence factor $\|E_{\rm TL}\|_{A}$~\cite[Theorem~4.1]{Falgout2005}; see~\cite{XZ2002,Zikatanov2008} for an abstract version.

\begin{theorem}\label{thm:old}
Let $\Pi_{A}$ be defined by~\eqref{piA}, and define
\begin{equation}\label{tildMs}
\widetilde{M}_{\rm s}:=M_{\rm s}^{T}\big(M_{\rm s}+M_{\rm s}^{T}-A_{\rm s}\big)^{-1}M_{\rm s}.
\end{equation}
Then, the convergence factor of Algorithm~{\rm\ref{alg:TL}} can be characterized as
\begin{equation}\label{old-id}
\|E_{\rm TL}\|_{A}=1-\frac{1}{K_{\rm TL}},
\end{equation}
where
\begin{equation}\label{KTL}
K_{\rm TL}=\sup_{\mathbf{v}\in\mathcal{R}(I-\Pi_{A})\backslash\{0\}}\,\inf_{\mathbf{v}_{\rm s}:\,\mathbf{v}=(I-\Pi_{A})S\mathbf{v}_{\rm s}}\frac{\mathbf{v}_{\rm s}^{T}\widetilde{M}_{\rm s}\mathbf{v}_{\rm s}}{\mathbf{v}^{T}A\mathbf{v}}.
\end{equation}
\end{theorem}

\begin{remark}
For a given $\mathbf{v}\in\mathcal{R}(I-\Pi_{A})\backslash\{0\}$, it is often tough to determine the set $\big\{\mathbf{v}_{\rm s}\in\mathbb{R}^{n_{\rm s}}:\mathbf{v}=(I-\Pi_{A})S\mathbf{v}_{\rm s}\big\}$, so the expression~\eqref{KTL} may limit the application of~\eqref{old-id}. For instance, it is difficult to derive a prolongation matrix for minimizing $K_{\rm TL}$ (or, equivalently, $\|E_{\rm TL}\|_{A}$), except for the special two-grid case where $n_{\rm s}=n$ and $S=I_{n}$ (see~\cite{XZ2017,Brannick2018}). This motivates us to establish an easy-to-use convergence identity for Algorithm~\ref{alg:TL}.
\end{remark}

\subsection{Reduction-based two-level methods} \label{subsec:TLr}

As proved in Theorem~\ref{thm:main}, the convergence factor $\|E_{\rm TL}\|_{A}$ can be expressed as
\begin{displaymath}
\|E_{\rm TL}\|_{A}=1-\lambda_{n_{\rm s}+n_{\rm c}-n+1}\big(\widetilde{M}_{\rm s}^{-1}S^{T}A(I-\Pi_{A})S\big),
\end{displaymath}
in which $S^{T}A(I-\Pi_{A})S\succeq 0$ is a Schur complement of the matrix
\begin{equation}\label{A-hat}
\widehat{A}:=\begin{pmatrix}
S^{T} \\
P^{T}
\end{pmatrix}A\big(S \,\ P\big)=\begin{pmatrix}
A_{\rm s} & S^{T}AP \\
P^{T}AS & A_{\rm c}
\end{pmatrix}.
\end{equation}
Since
\begin{displaymath}
\lambda\big(\widetilde{M}_{\rm s}^{-1}S^{T}A(I-\Pi_{A})S\big)=\lambda\big(\widetilde{M}_{\rm s}^{-\frac{1}{2}}S^{T}A(I-\Pi_{A})S\widetilde{M}_{\rm s}^{-\frac{1}{2}}\big)
\end{displaymath}
and
\begin{displaymath}
\widetilde{M}_{\rm s}^{-\frac{1}{2}}A_{\rm s}\widetilde{M}_{\rm s}^{-\frac{1}{2}}\succeq\widetilde{M}_{\rm s}^{-\frac{1}{2}}S^{T}A(I-\Pi_{A})S\widetilde{M}_{\rm s}^{-\frac{1}{2}},
\end{displaymath}
we obtain
\begin{displaymath}
\|E_{\rm TL}\|_{A}\geq 1-\lambda_{n_{\rm s}+n_{\rm c}-n+1}\big(\widetilde{M}_{\rm s}^{-\frac{1}{2}}A_{\rm s}\widetilde{M}_{\rm s}^{-\frac{1}{2}}\big)=1-\lambda_{n_{\rm s}+n_{\rm c}-n+1}\big(\widetilde{M}_{\rm s}^{-1}A_{\rm s}\big).
\end{displaymath}
To attain the lower bound $1-\lambda_{n_{\rm s}+n_{\rm c}-n+1}\big(\widetilde{M}_{\rm s}^{-1}A_{\rm s}\big)$, it suffices to set
\begin{equation}\label{SAP}
S^{T}AP=0,
\end{equation}
that is, the hierarchical spaces $\mathcal{R}(S)$ and $\mathcal{R}(P)$ are orthogonal with respect to the $A$-inner product. If, in addition, $M_{\rm s}=A_{\rm s}$, then the convergence factor is exactly zero. This theoretical result can serve as a motivation for developing some practical algorithms.

Let $A$ be partitioned into the two-by-two block form
\begin{equation}\label{two-by-two}
A=\begin{pmatrix}
A_{\rm ff} & A_{\rm fc} \\
A_{\rm cf} & A_{\rm cc}
\end{pmatrix},
\end{equation}
where $A_{\rm ff}\in\mathbb{R}^{n_{\rm f}\times n_{\rm f}}$, $A_{\rm fc}\in\mathbb{R}^{n_{\rm f}\times n_{\rm c}}$, $A_{\rm cf}=A_{\rm fc}^{T}$, $A_{\rm cc}\in\mathbb{R}^{n_{\rm c}\times n_{\rm c}}$, and $n_{\rm f}+n_{\rm c}=n$. Take
\begin{displaymath}
S=\begin{pmatrix}
I_{n_{\rm f}} \\
0
\end{pmatrix} \quad \text{and} \quad P=\begin{pmatrix}
W_{\rm fc} \\
I_{n_{\rm c}}
\end{pmatrix},
\end{displaymath}
where $W_{\rm fc}\in\mathbb{R}^{n_{\rm f}\times n_{\rm c}}$. Then, the relation~\eqref{SAP} yields
\begin{displaymath}
W_{\rm fc}=-A_{\rm ff}^{-1}A_{\rm fc}.
\end{displaymath}
The resulting prolongation matrix, denoted by $P_{\star}$, is of the form
\begin{displaymath}
P_{\star}=\begin{pmatrix}
-A_{\rm ff}^{-1}A_{\rm fc} \\
I_{n_{\rm c}}
\end{pmatrix},
\end{displaymath}
which is commonly called an \textit{ideal} prolongation matrix~\cite{Falgout2004,XXF2018}. In Algorithm~\ref{alg:TL}, if
\begin{displaymath}
S=\begin{pmatrix}
I_{n_{\rm f}} \\
0
\end{pmatrix}, \quad M_{\rm s}=A_{\rm ff}, \quad \text{and} \quad P=P_{\star},
\end{displaymath}
then the corresponding convergence factor is zero. However, from a computational point of view, such an algorithm may not be practical, because it is often too costly to compute $A_{\rm ff}^{-1}$ directly and $A_{\rm ff}^{-1}$ is generally dense. Hence, a sparse approximation to $A_{\rm ff}^{-1}$ is needed to design a practical algorithm. Furthermore, the convergence of the resulting algorithm is expected to be guaranteed theoretically.

An approximation algorithm was proposed and studied in~\cite{MacLachlan2006}. Since the convergence factor of Algorithm~\ref{alg:TL} is the square of that of Algorithm~\ref{alg:TL} without postsmoothing, the convergence estimates in~\cite[Theorem~1 and Corollary~1]{MacLachlan2006} can be formalized as the following theorem.

\begin{theorem}\label{thm:TLr0}
Let $A$ be partitioned as in~\eqref{two-by-two}, and let $D_{\rm ff}\in\mathbb{R}^{n_{\rm f}\times n_{\rm f}}$ be an SPD matrix such that
\begin{displaymath}
A_{\rm ff}\succeq D_{\rm ff}\succeq \frac{1}{1+\varepsilon}A_{\rm ff} \quad \text{and} \quad \begin{pmatrix}
D_{\rm ff} & A_{\rm fc} \\
A_{\rm cf} & A_{\rm cc}
\end{pmatrix}\succeq 0,
\end{displaymath}
where $\varepsilon>0$ is a parameter. Take
\begin{displaymath}
S=\begin{pmatrix}
I_{n_{\rm f}} \\
0
\end{pmatrix}, \quad M_{\rm s}=\bigg(1+\frac{\varepsilon}{2}\bigg)D_{\rm ff}, \quad \text{and} \quad P=\begin{pmatrix}
-D_{\rm ff}^{-1}A_{\rm fc} \\
I_{n_{\rm c}}
\end{pmatrix}.
\end{displaymath}
Then, the convergence factor of Algorithm~{\rm\ref{alg:TL}} satisfies
\begin{equation}\label{old-TLr1}
\|E_{\rm TL}\|_{A}\leq\frac{\varepsilon}{1+\varepsilon}\bigg(1+\frac{\varepsilon}{(2+\varepsilon)^{2}}\bigg).
\end{equation}
More generally, if the pre- and postsmoothing steps in Algorithm~{\rm\ref{alg:TL}} are carried out $\nu$ times iteratively, then
\begin{equation}\label{old-TLr2}
\|E_{\rm TL}\|_{A}\leq\frac{\varepsilon}{1+\varepsilon}\bigg(1+\frac{\varepsilon^{2\nu-1}}{(2+\varepsilon)^{2\nu}}\bigg).
\end{equation}
\end{theorem}

\begin{remark}
The convergence estimates \eqref{old-TLr1} and~\eqref{old-TLr2} can be improved by using a new characterization of $\|E_{\rm TL}\|_{A}$, which will be discussed in the next section.
\end{remark}

\section{A new characterization of $\|E_{\rm TL}\|_{A}$ and its applications} \label{sec:main}

\subsection{The new characterization}

We first give a new characterization of $\|E_{\rm TL}\|_{A}$, which plays a fundamental role in the subsequent analysis.

\begin{theorem}\label{thm:main}
Under the assumptions of Algorithm~{\rm\ref{alg:TL}}, it holds that
\begin{equation}\label{new-id}
\|E_{\rm TL}\|_{A}=1-\lambda_{n_{\rm s}+n_{\rm c}-n+1}\big(\widetilde{M}_{\rm s}^{-1}S^{T}A(I-\Pi_{A})S\big),
\end{equation}
where $\Pi_{A}$ and $\widetilde{M}_{\rm s}$ are defined by~\eqref{piA} and~\eqref{tildMs}, respectively.
\end{theorem}

\begin{proof}
From~\eqref{ETL1} and~\eqref{ETL2}, we have
\begin{displaymath}
B_{\rm TL}^{-1}A=I-\big(I-SM_{\rm s}^{-T}S^{T}A\big)(I-\Pi_{A})\big(I-SM_{\rm s}^{-1}S^{T}A\big).
\end{displaymath}
Then
\begin{align*}
\lambda\big(B_{\rm TL}^{-1}A\big)&=\lambda\big(I-\big(I-SM_{\rm s}^{-T}S^{T}A\big)(I-\Pi_{A})\big(I-SM_{\rm s}^{-1}S^{T}A\big)\big)\\
&=\lambda\big(I-\big(I-SM_{\rm s}^{-1}S^{T}A\big)\big(I-SM_{\rm s}^{-T}S^{T}A\big)(I-\Pi_{A})\big)\\
&=\lambda\big(I-\big(I-S\widetilde{M}_{\rm s}^{-1}S^{T}A\big)(I-\Pi_{A})\big)\\
&=\lambda\big(S\widetilde{M}_{\rm s}^{-1}S^{T}A(I-\Pi_{A})+\Pi_{A}\big).
\end{align*}
Since $\Pi_{A}$ is a projector (i.e., $\Pi_{A}^{2}=\Pi_{A}$) and $\rank(\Pi_{A})=n_{\rm c}$, there exists a nonsingular matrix $X\in\mathbb{R}^{n\times n}$ such that
\begin{displaymath}
\Pi_{A}=X^{-1}\begin{pmatrix}
I_{n_{\rm c}} & 0 \\
0 & 0
\end{pmatrix}X.
\end{displaymath}
Let
\begin{displaymath}
S\widetilde{M}_{\rm s}^{-1}S^{T}A=X^{-1}\begin{pmatrix}
Y_{11} & Y_{12} \\
Y_{21} & Y_{22}
\end{pmatrix}X,
\end{displaymath}
where $Y_{11}\in\mathbb{R}^{n_{\rm c}\times n_{\rm c}}$, $Y_{12}\in\mathbb{R}^{n_{\rm c}\times(n-n_{\rm c})}$, $Y_{21}\in\mathbb{R}^{(n-n_{\rm c})\times n_{\rm c}}$, and $Y_{22}\in\mathbb{R}^{(n-n_{\rm c})\times(n-n_{\rm c})}$. Direct computations yield
\begin{align}
S\widetilde{M}_{\rm s}^{-1}S^{T}A(I-\Pi_{A})&=X^{-1}\begin{pmatrix}
0 & Y_{12} \\
0 & Y_{22}
\end{pmatrix}X,\label{sim-1}\\
S\widetilde{M}_{\rm s}^{-1}S^{T}A(I-\Pi_{A})+\Pi_{A}&=X^{-1}\begin{pmatrix}
I_{n_{\rm c}} & Y_{12} \\
0 & Y_{22}
\end{pmatrix}X.\label{sim-2}
\end{align}
Recall that
\begin{displaymath}
I\succeq A^{\frac{1}{2}}B_{\rm TL}^{-1}A^{\frac{1}{2}}\succ 0.
\end{displaymath}
We then have
\begin{displaymath}
\lambda\big(S\widetilde{M}_{\rm s}^{-1}S^{T}A(I-\Pi_{A})+\Pi_{A}\big)=\lambda\big(B_{\rm TL}^{-1}A\big)=\lambda\big(A^{\frac{1}{2}}B_{\rm TL}^{-1}A^{\frac{1}{2}}\big)\subset(0,1],
\end{displaymath}
which, combined with~\eqref{sim-2}, leads to
\begin{displaymath}
\lambda(Y_{22})\subset(0,1].
\end{displaymath}
By~\eqref{sim-1} and~\eqref{sim-2}, we have
\begin{align*}
\lambda_{\min}\big(B_{\rm TL}^{-1}A\big)&=\lambda_{\min}\big(S\widetilde{M}_{\rm s}^{-1}S^{T}A(I-\Pi_{A})+\Pi_{A}\big)\\
&=\lambda_{\min}(Y_{22})\\
&=\lambda_{n_{\rm c}+1}\big(S\widetilde{M}_{\rm s}^{-1}S^{T}A(I-\Pi_{A})\big),
\end{align*}
which, together with the fact
\begin{displaymath}
\lambda\big(S\widetilde{M}_{\rm s}^{-1}S^{T}A(I-\Pi_{A})\big)=\{\underbrace{0,\ldots,0}_{n-n_{\rm s}}\}\cup\lambda\big(\widetilde{M}_{\rm s}^{-1}S^{T}A(I-\Pi_{A})S\big),
\end{displaymath}
yields
\begin{displaymath}
\lambda_{\min}\big(B_{\rm TL}^{-1}A\big)=\lambda_{n_{\rm s}+n_{\rm c}-n+1}\big(\widetilde{M}_{\rm s}^{-1}S^{T}A(I-\Pi_{A})S\big).
\end{displaymath}
The identity~\eqref{new-id} then follows immediately by using~\eqref{norm-ETL}.
\end{proof}

\begin{remark}
In fact, the quantity $n_{\rm s}+n_{\rm c}-n$ involved in~\eqref{new-id} is the dimension of $\mathcal{R}(S)\cap\mathcal{R}(P)$, because
\begin{displaymath}
\dim\big(\mathcal{R}(S)\big)+\dim\big(\mathcal{R}(P)\big)=\dim\big(\mathcal{R}(S)+\mathcal{R}(P)\big)+\dim\big(\mathcal{R}(S)\cap\mathcal{R}(P)\big),
\end{displaymath}
where $\dim(\cdot)$ denotes the dimension of a subspace of $\mathbb{R}^{n}$.
\end{remark}

\begin{remark}
In view of~\eqref{A-hat}, we define
\begin{displaymath}
\gamma:=\max_{\substack{\mathbf{v}_{\rm s}\in\mathbb{R}^{n_{\rm s}}\backslash\{0\} \\ \mathbf{v}_{\rm c}\in\mathbb{R}^{n_{\rm c}}\backslash\{0\}}}\frac{\mathbf{v}_{\rm s}^{T}S^{T}AP\mathbf{v}_{\rm c}}{\sqrt{\mathbf{v}_{\rm s}^{T}A_{\rm s}\mathbf{v}_{\rm s}\cdot\mathbf{v}_{\rm c}^{T}A_{\rm c}\mathbf{v}_{\rm c}}},
\end{displaymath}
which can be equivalently expressed as
\begin{equation}\label{CBS}
\gamma=\big\|A_{\rm s}^{-\frac{1}{2}}S^{T}APA_{\rm c}^{-\frac{1}{2}}\big\|_{2}=\sqrt{\lambda_{\max}\big(A_{\rm s}^{-\frac{1}{2}}S^{T}APA_{\rm c}^{-1}P^{T}ASA_{\rm s}^{-\frac{1}{2}}\big)}.
\end{equation}
The assumption~\eqref{rank} entails that 
\begin{displaymath}
n_{\rm s}+n_{\rm c}\geq n.
\end{displaymath}
Observe from~\eqref{A-hat} that $\widehat{A}\succeq 0$, and $\widehat{A}\succ 0$ if and only if $n_{\rm s}+n_{\rm c}=n$, i.e., $(S \,\ P)$ is square and nonsingular.
\begin{itemize}

\item If $n_{\rm s}+n_{\rm c}>n$, then $\widehat{A}\succeq 0$ is singular, which leads to the positive semidefiniteness and singularity of the Schur complement $A_{\rm s}-S^{T}APA_{\rm c}^{-1}P^{T}AS$. Then
\begin{displaymath}
\lambda_{\min}\big(I_{n_{\rm s}}-A_{\rm s}^{-\frac{1}{2}}S^{T}APA_{\rm c}^{-1}P^{T}ASA_{\rm s}^{-\frac{1}{2}}\big)=0
\end{displaymath}
and hence
\begin{displaymath}
\gamma=\sqrt{\lambda_{\max}\big(A_{\rm s}^{-\frac{1}{2}}S^{T}APA_{\rm c}^{-1}P^{T}ASA_{\rm s}^{-\frac{1}{2}}\big)}=1.
\end{displaymath}

\item If $n_{\rm s}+n_{\rm c}=n$, then $\gamma\in[0,1)$, which is the so-called \textit{Cauchy--Bunyakowski--Schwarz} (C.B.S.) constant; see, e.g.,~\cite{Axelsson1994,Axelsson2003}. In such a case, $\gamma$ can be viewed as the cosine of the abstract angle between $\mathcal{R}(S)$ and $\mathcal{R}(P)$ with respect to the $A$-inner product.

\end{itemize}
From~\eqref{CBS}, we deduce that
\begin{displaymath}
\gamma^{2}A_{\rm s}\succeq S^{T}A\Pi_{A}S.
\end{displaymath}
Then
\begin{equation}\label{SApi-low1}
S^{T}A(I-\Pi_{A})S\succeq(1-\gamma^{2})A_{\rm s},
\end{equation}
which, combined with~\eqref{new-id}, gives
\begin{displaymath}
\|E_{\rm TL}\|_{A}\leq 1-(1-\gamma^{2})\lambda_{n_{\rm s}+n_{\rm c}-n+1}\big(\widetilde{M}_{\rm s}^{-1}A_{\rm s}\big).
\end{displaymath}
Note that the above estimate is of interest only when $n_{\rm s}+n_{\rm c}=n$, because $\gamma=1$ if $n_{\rm s}+n_{\rm c}>n$. In the case $n_{\rm s}+n_{\rm c}=n$, the above estimate reduces to
\begin{equation}\label{ETG-up}
\|E_{\rm TL}\|_{A}\leq 1-(1-\gamma^{2})\lambda_{\min}\big(\widetilde{M}_{\rm s}^{-1}A_{\rm s}\big).
\end{equation}
Further, if $M_{\rm s}=A_{\rm s}$, then the upper bound in~\eqref{ETG-up} will become $\gamma^{2}$, and~\eqref{ETG-up} will become an equality. Indeed,
\begin{align*}
\|E_{\rm TL}\|_{A}&=1-\lambda_{\min}\big(A_{\rm s}^{-1}S^{T}A(I-\Pi_{A})S\big)\\
&=1-\lambda_{\min}\big(I_{n_{\rm s}}-A_{\rm s}^{-1}S^{T}APA_{\rm c}^{-1}P^{T}AS\big)\\
&=\lambda_{\max}\big(A_{\rm s}^{-1}S^{T}APA_{\rm c}^{-1}P^{T}AS\big)\\
&=\gamma^{2}.
\end{align*}
\end{remark}

From the proof of Theorem~\ref{thm:main}, we can get a description of the spectrum $\lambda(E_{\rm TL})$, as stated in the following corollary.

\begin{corollary}
Under the same assumptions as in Theorem~{\rm\ref{thm:main}}, one has
\begin{displaymath}
\lambda(E_{\rm TL})=\{\underbrace{0,\ldots,0}_{n_{\rm c}},1-\nu_{1},\ldots,1-\nu_{n-n_{\rm c}}\},
\end{displaymath}
where $\{\nu_{i}\}_{i=1}^{n-n_{\rm c}}$ are the positive eigenvalues of $\widetilde{M}_{\rm s}^{-1}S^{T}A(I-\Pi_{A})S$.
\end{corollary}

\begin{proof}
By~\eqref{sim-2}, we have
\begin{displaymath}
\lambda\big(B_{\rm TL}^{-1}A\big)=\lambda\big(S\widetilde{M}_{\rm s}^{-1}S^{T}A(I-\Pi_{A})+\Pi_{A}\big)=\{\underbrace{1,\ldots,1}_{n_{\rm c}}\}\cup\lambda(Y_{22}),
\end{displaymath}
where $\lambda(Y_{22})\subset(0,1]$. Hence,
\begin{displaymath}
\lambda(E_{\rm TL})=\lambda\big(I-B_{\rm TL}^{-1}A\big)=\{\underbrace{0,\ldots,0}_{n_{\rm c}}\}\cup\lambda(I_{n-n_{\rm c}}-Y_{22}).
\end{displaymath}
The desired result then follows from~\eqref{sim-1} and the basic fact that $S\widetilde{M}_{\rm s}^{-1}S^{T}A(I-\Pi_{A})$ and $\widetilde{M}_{\rm s}^{-1}S^{T}A(I-\Pi_{A})S$ have the same nonzero eigenvalues.
\end{proof}

The identity~\eqref{new-id} is more convenient for analyzing two-level methods compared to~\eqref{old-id}. To illustrate its usability, we next show three applications of~\eqref{new-id}.

\subsection{The first application}

The first application of~\eqref{new-id} is to study the influence of $\mathcal{R}(P)$ on $\|E_{\rm TL}\|_{A}$.

The following lemma gives a simple but useful eigenvalue inequality; see, e.g.,~\cite[Corollary~4.3.5]{Horn2013}.

\begin{lemma}\label{lem:eig-ineq}
Let $H_{1}$ and $H_{2}$ be Hermitian matrices of order $n$. If $H_{2}$ is singular, then
\begin{displaymath}
\lambda_{i}(H_{1}+H_{2})\leq\lambda_{i+\rank(H_{2})}(H_{1})
\end{displaymath}
for all $i=1,\ldots,n-\rank(H_{2})$.
\end{lemma}

Using~\eqref{new-id} and Lemma~\ref{lem:eig-ineq}, we can prove that the two-level convergence factor decreases as the hierarchical space $\mathcal{R}(P)$ expands, as described below.

\begin{theorem}
Assume that $\widehat{P}\in\mathbb{R}^{n\times\hat{n}_{\rm c}} \ (n_{\rm c}\leq\hat{n}_{\rm c}<n)$ is of full column rank and $\rank(S \ \widehat{P})=n$. Define
\begin{displaymath}
\widehat{\Pi}_{A}:=\widehat{P}(\widehat{P}^{T}A\widehat{P})^{-1}\widehat{P}^{T}A.
\end{displaymath}
If $\mathcal{R}(P)\subseteq\mathcal{R}(\widehat{P})$, then
\begin{displaymath}
\|\widehat{E}_{\rm TL}\|_{A}\leq\|E_{\rm TL}\|_{A},
\end{displaymath}
where
\begin{equation}\label{hat-ETL}
\|\widehat{E}_{\rm TL}\|_{A}=1-\lambda_{n_{\rm s}+\hat{n}_{\rm c}-n+1}\big(\widetilde{M}_{\rm s}^{-1}S^{T}A(I-\widehat{\Pi}_{A})S\big)
\end{equation}
and $\|E_{\rm TL}\|_{A}$ is given by~\eqref{new-id}.
\end{theorem}

\begin{proof}
Since $\mathcal{R}(P)\subseteq\mathcal{R}(\widehat{P})$, there exists an $\hat{n}_{\rm c}\times n_{\rm c}$ matrix $Z$ such that
\begin{displaymath}
P=\widehat{P}Z,
\end{displaymath}
from which we can deduce that $Z$ is of full column rank. Furthermore, there exists a nonsingular matrix $\widehat{Z}\in\mathbb{R}^{\hat{n}_{\rm c}\times\hat{n}_{\rm c}}$ such that
\begin{displaymath}
Z=\widehat{Z}\begin{pmatrix}
I_{n_{\rm c}} \\
0
\end{pmatrix},
\end{displaymath}
which yields
\begin{displaymath}
P=\widehat{P}\widehat{Z}\begin{pmatrix}
I_{n_{\rm c}} \\
0
\end{pmatrix}.
\end{displaymath}
Hence,
\begin{displaymath}
\widehat{P}=(P \,\ Q)\widehat{Z}^{-1}
\end{displaymath}
for some $Q\in\mathbb{R}^{n\times(\hat{n}_{\rm c}-n_{\rm c})}$, and~\eqref{hat-ETL} can be equivalently expressed as
\begin{displaymath}
\|\widehat{E}_{\rm TL}\|_{A}=1-\lambda_{n_{\rm s}+\hat{n}_{\rm c}-n+1}\big(\widetilde{M}_{\rm s}^{-1}S^{T}A\big(I-\widehat{P}_{0}(\widehat{P}_{0}^{T}A\widehat{P}_{0})^{-1}\widehat{P}_{0}^{T}A\big)S\big)
\end{displaymath}
with $\widehat{P}_{0}=(P \,\ Q)$.

Let
\begin{displaymath}
F=\widetilde{M}_{\rm s}^{-\frac{1}{2}}S^{T}A(I-\Pi_{A})S\widetilde{M}_{\rm s}^{-\frac{1}{2}}-\widetilde{M}_{\rm s}^{-\frac{1}{2}}S^{T}A\big(I-\widehat{P}_{0}(\widehat{P}_{0}^{T}A\widehat{P}_{0})^{-1}\widehat{P}_{0}^{T}A\big)S\widetilde{M}_{\rm s}^{-\frac{1}{2}}.
\end{displaymath}
Then
\begin{align*}
F&=\widetilde{M}_{\rm s}^{-\frac{1}{2}}S^{T}A\big(\widehat{P}_{0}(\widehat{P}_{0}^{T}A\widehat{P}_{0})^{-1}\widehat{P}_{0}^{T}-PA_{\rm c}^{-1}P^{T}\big)AS\widetilde{M}_{\rm s}^{-\frac{1}{2}}\\
&=\widetilde{M}_{\rm s}^{-\frac{1}{2}}S^{T}A\bigg(\widehat{P}_{0}(\widehat{P}_{0}^{T}A\widehat{P}_{0})^{-1}\widehat{P}_{0}^{T}-\widehat{P}_{0}\begin{pmatrix}
A_{\rm c}^{-1} & 0 \\
0 & 0
\end{pmatrix}\widehat{P}_{0}^{T}\bigg)AS\widetilde{M}_{\rm s}^{-\frac{1}{2}}\\
&=\widetilde{M}_{\rm s}^{-\frac{1}{2}}S^{T}A\widehat{P}_{0}\bigg(\underbrace{(\widehat{P}_{0}^{T}A\widehat{P}_{0})^{-1}-\begin{pmatrix}
A_{\rm c}^{-1} & 0 \\
0 & 0
\end{pmatrix}}_{=:F_{0}}\bigg)\widehat{P}_{0}^{T}AS\widetilde{M}_{\rm s}^{-\frac{1}{2}}.
\end{align*}
Direct computation yields
\begin{displaymath}
F_{0}=\begin{pmatrix}
-A_{\rm c}^{-1}P^{T}AQ \\
I_{\hat{n}_{\rm c}-n_{\rm c}}
\end{pmatrix}\big(Q^{T}AQ-Q^{T}APA_{\rm c}^{-1}P^{T}AQ\big)^{-1}\begin{pmatrix}
-Q^{T}APA_{\rm c}^{-1} & I_{\hat{n}_{\rm c}-n_{\rm c}}
\end{pmatrix}.
\end{displaymath}
Thus,
\begin{displaymath}
\rank(F)\leq\rank(F_{0})=\hat{n}_{\rm c}-n_{\rm c}.
\end{displaymath}
Obviously, $F\in\mathbb{R}^{n_{\rm s}\times n_{\rm s}}$ is singular, because
\begin{displaymath}
\rank(F)\leq\hat{n}_{\rm c}-n_{\rm c}<n-n_{\rm c}\leq n_{\rm s}.
\end{displaymath}
Using~\eqref{new-id} and Lemma~\ref{lem:eig-ineq}, we obtain
\begin{align*}
\|E_{\rm TL}\|_{A}&=1-\lambda_{n_{\rm s}+n_{\rm c}-n+1}\big(\widetilde{M}_{\rm s}^{-\frac{1}{2}}S^{T}A(I-\Pi_{A})S\widetilde{M}_{\rm s}^{-\frac{1}{2}}\big)\\
&=1-\lambda_{n_{\rm s}+n_{\rm c}-n+1}\big(\widetilde{M}_{\rm s}^{-\frac{1}{2}}S^{T}A\big(I-\widehat{P}_{0}(\widehat{P}_{0}^{T}A\widehat{P}_{0})^{-1}\widehat{P}_{0}^{T}A\big)S\widetilde{M}_{\rm s}^{-\frac{1}{2}}+F\big)\\
&\geq 1-\lambda_{n_{\rm s}+n_{\rm c}-n+1+\rank(F)}\big(\widetilde{M}_{\rm s}^{-\frac{1}{2}}S^{T}A\big(I-\widehat{P}_{0}(\widehat{P}_{0}^{T}A\widehat{P}_{0})^{-1}\widehat{P}_{0}^{T}A\big)S\widetilde{M}_{\rm s}^{-\frac{1}{2}}\big)\\
&\geq 1-\lambda_{n_{\rm s}+\hat{n}_{\rm c}-n+1}\big(\widetilde{M}_{\rm s}^{-\frac{1}{2}}S^{T}A\big(I-\widehat{P}_{0}(\widehat{P}_{0}^{T}A\widehat{P}_{0})^{-1}\widehat{P}_{0}^{T}A\big)S\widetilde{M}_{\rm s}^{-\frac{1}{2}}\big)\\
&=\|\widehat{E}_{\rm TL}\|_{A},
\end{align*}
which gives the desired result.
\end{proof}

\subsection{The second application}

The second application of~\eqref{new-id} is to find an optimal prolongation matrix for minimizing the convergence factor $\|E_{\rm TL}\|_{A}$, provided that $S$ and $M_{\rm s}$ are preselected.

To study the optimality of prolongation matrix, we need the following inequality, which is known as the \textit{Poincar\'{e} separation theorem}; see, e.g.,~\cite[Corollary~4.3.37]{Horn2013}.

\begin{lemma}\label{lem:Poincare}
Let $H\in\mathbb{C}^{n\times n}$ be Hermitian, and let $\{\mathbf{q}_{k}\}_{k=1}^{m}\subset\mathbb{C}^{n}\,(1\leq m\leq n)$ be a set of orthonormal vectors. Then, for any $i=1,\ldots,m$, it holds that
\begin{displaymath}
\lambda_{i}(H)\leq\lambda_{i}(\check{H})\leq\lambda_{i+n-m}(H),
\end{displaymath}
where $\check{H}=\big(\mathbf{q}_{i}^{\ast}H\mathbf{q}_{j}\big)\in\mathbb{C}^{m\times m}$ and $(\cdot)^{\ast}$ denotes the conjugate transpose of a vector.
\end{lemma}

Based on~\eqref{new-id} and the above lemma, we can derive a class of optimal prolongation matrices, as described in the following theorem.

\begin{theorem}\label{thm:optimal}
Let $\{(\mu_{i},\mathbf{v}_{i})\}_{i=1}^{n}$ be the eigenpairs of $S\widetilde{M}_{\rm s}^{-1}S^{T}A$, namely,
\begin{displaymath}
S\widetilde{M}_{\rm s}^{-1}S^{T}A\mathbf{v}_{i}=\mu_{i}\mathbf{v}_{i},
\end{displaymath}
where
\begin{displaymath}
0\leq\mu_{1}\leq\mu_{2}\leq\cdots\leq\mu_{n}\leq 1 \quad \text{and} \quad \mathbf{v}_{i}^{T}A\mathbf{v}_{j}=\begin{cases}
1 & \text{if $i=j$},\\
0 & \text{if $i\neq j$}.
\end{cases}
\end{displaymath}
Then
\begin{displaymath}
\|E_{\rm TL}\|_{A}\geq 1-\mu_{n_{\rm c}+1},
\end{displaymath}
and the equality holds if $\mathcal{R}(P)=\Span\{\mathbf{v}_{1},\ldots,\mathbf{v}_{n_{\rm c}}\}$.
\end{theorem}

\begin{proof}
Since $S\widetilde{M}_{\rm s}^{-1}S^{T}A$ has the same nonzero eigenvalues as $\widetilde{M}_{\rm s}^{-1}A_{\rm s}$ and 
\begin{displaymath}
\widetilde{M}_{\rm s}\succeq A_{\rm s}\succ 0,
\end{displaymath}
we conclude that
\begin{displaymath}
0=\mu_{1}=\cdots=\mu_{n-n_{\rm s}}<\mu_{n-n_{\rm s}+1}\leq\cdots\leq\mu_{n}\leq 1.
\end{displaymath}
Let
\begin{displaymath}
V=(\mathbf{v}_{1},\ldots,\mathbf{v}_{n}) \quad \text{and} \quad U_{1}=V^{-1}P(P^{T}V^{-T}V^{-1}P)^{-\frac{1}{2}}.
\end{displaymath}
It is easy to verify that $V^{T}AV=I$ and $U_{1}\in\mathbb{R}^{n\times n_{\rm c}}$ has orthonormal columns (i.e., $U_{1}^{T}U_{1}=I_{n_{\rm c}}$). Let $U_{2}$ be an $n\times(n-n_{\rm c})$ matrix such that $(U_{1} \,\ U_{2})\in\mathbb{R}^{n\times n}$ is an orthogonal matrix. Then
\begin{align*}
S\widetilde{M}_{\rm s}^{-1}S^{T}A(I-\Pi_{A})&=S\widetilde{M}_{\rm s}^{-1}S^{T}A\big(I-PA_{\rm c}^{-1}P^{T}A\big)\\
&=S\widetilde{M}_{\rm s}^{-1}S^{T}A\big(I-VU_{1}U_{1}^{T}V^{T}A\big)\\
&=S\widetilde{M}_{\rm s}^{-1}S^{T}A\big(I-VU_{1}U_{1}^{T}V^{-1}\big)\\
&=S\widetilde{M}_{\rm s}^{-1}S^{T}AV\big(I-U_{1}U_{1}^{T}\big)V^{-1}\\
&=V\Lambda U_{2}U_{2}^{T}V^{-1},
\end{align*}
where
\begin{displaymath}
\Lambda=\diag(0,\ldots,0,\mu_{n-n_{\rm s}+1},\ldots,\mu_{n}).
\end{displaymath}
Since $\Lambda U_{2}U_{2}^{T}$ is similar to $S\widetilde{M}_{\rm s}^{-1}S^{T}A(I-\Pi_{A})$ and $S\widetilde{M}_{\rm s}^{-1}S^{T}A(I-\Pi_{A})$ has $n-n_{\rm c}$ positive eigenvalues (see~\eqref{sim-1}), it follows that $\Lambda U_{2}U_{2}^{T}$ has $n-n_{\rm c}$ positive eigenvalues, which leads to $U_{2}^{T}\Lambda U_{2}\succ 0$. By~\eqref{new-id} and Lemma~\ref{lem:Poincare}, we have
\begin{align*}
\|E_{\rm TL}\|_{A}&=1-\lambda_{n_{\rm s}+n_{\rm c}-n+1}\big(\widetilde{M}_{\rm s}^{-1}S^{T}A(I-\Pi_{A})S\big)\\
&=1-\lambda_{n_{\rm c}+1}\big(S\widetilde{M}_{\rm s}^{-1}S^{T}A(I-\Pi_{A})\big)\\
&=1-\lambda_{n_{\rm c}+1}(\Lambda U_{2}U_{2}^{T})\\
&=1-\lambda_{1}(U_{2}^{T}\Lambda U_{2})\\
&\geq 1-\lambda_{n_{\rm c}+1}(\Lambda)\\
&=1-\mu_{n_{\rm c}+1}.
\end{align*}

In particular, if 
\begin{displaymath}
\mathcal{R}(P)=\Span\{\mathbf{v}_{1},\ldots,\mathbf{v}_{n_{\rm c}}\},
\end{displaymath}
then there exists a nonsingular matrix $P_{\rm c}\in\mathbb{R}^{n_{\rm c}\times n_{\rm c}}$ such that
\begin{displaymath}
P=V\begin{pmatrix}
P_{\rm c} \\
0
\end{pmatrix}.
\end{displaymath}
We then have
\begin{displaymath}
U_{1}=\begin{pmatrix}
P_{\rm c} \\
0
\end{pmatrix}(P_{\rm c}^{T}P_{\rm c})^{-\frac{1}{2}} \quad \text{and} \quad U_{2}U_{2}^{T}=I-U_{1}U_{1}^{T}=\begin{pmatrix}
0 & 0 \\
0 & I_{n-n_{\rm c}}
\end{pmatrix}.
\end{displaymath}
Hence,
\begin{displaymath}
\|E_{\rm TL}\|_{A}=1-\lambda_{n_{\rm c}+1}(\Lambda U_{2}U_{2}^{T})=1-\mu_{n_{\rm c}+1}.
\end{displaymath}
This completes the proof.
\end{proof}

\begin{remark}
For the special \textit{two-grid} case (i.e., $n_{\rm s}=n$), the corresponding optimal prolongation theory can be found in~\cite{XZ2017,Brannick2018}, in which the inner product induced by a symmetrized smoother is used. In the general case $n_{\rm s}<n$, $S\widetilde{M}_{\rm s}^{-1}S^{T}$ is a singular matrix, that is, its inverse does not exist and hence cannot induce an inner product in $\mathbb{R}^{n}$ (so does $\widetilde{M}_{\rm s}$, because $n_{\rm s}<n$).
\end{remark}

\subsection{The third application}

The third application of~\eqref{new-id} is to analyze the convergence of a class of reduction-based two-level methods (see Subsection~\ref{subsec:TLr}).

The following convergence estimates generalize and improve the existing ones in Theorem~\ref{thm:TLr0}.

\begin{theorem}\label{thm:TLr}
Let $A$ be partitioned as in~\eqref{two-by-two}, and let $D_{\rm ff}\in\mathbb{R}^{n_{\rm f}\times n_{\rm f}}$ be an SPD matrix such that
\begin{displaymath}
A_{\rm ff}\succeq D_{\rm ff}\succeq\frac{1}{1+\varepsilon}A_{\rm ff} \quad \text{and} \quad \begin{pmatrix}
D_{\rm ff} & A_{\rm fc} \\
A_{\rm cf} & A_{\rm cc}
\end{pmatrix}\succeq 0,
\end{displaymath}
where $\varepsilon>0$ is a parameter. Take
\begin{displaymath}
S=\begin{pmatrix}
I_{n_{\rm f}} \\
0
\end{pmatrix}, \quad M_{\rm s}=\omega D_{\rm ff}, \quad \text{and} \quad P=\begin{pmatrix}
-D_{\rm ff}^{-1}A_{\rm fc} \\
I_{n_{\rm c}}
\end{pmatrix},
\end{displaymath}
where $\omega>\frac{1}{2}(1+\varepsilon)$. Then, the convergence factor of Algorithm~{\rm\ref{alg:TL}} satisfies
\begin{equation}\label{new-TLr1}
\|E_{\rm TL}\|_{A}\leq 1-\omega^{-2}(2\omega-1-\varepsilon).
\end{equation}
More generally, if the pre- and postsmoothing steps in Algorithm~{\rm\ref{alg:TL}} are carried out $\nu$ times iteratively, then
\begin{equation}\label{new-TLr2}
\|E_{\rm TL}\|_{A}\leq 1-\frac{1-\big(1-\omega^{-1}(1+\varepsilon)\big)^{2\nu}}{1+\varepsilon}.
\end{equation}
Moreover, the upper bounds in~\eqref{new-TLr1} and~\eqref{new-TLr2} attain the minimum $\frac{\varepsilon}{1+\varepsilon}$ if and only if $\omega=1+\varepsilon$.
\end{theorem}

\begin{proof}
(i) Let
\begin{displaymath}
\Delta_{\rm ff}=A_{\rm ff}-D_{\rm ff}.
\end{displaymath}
Direct computations yield
\begin{align*}
\widetilde{M}_{\rm s}^{-1}&=\omega^{-2}D_{\rm ff}^{-1}(2\omega D_{\rm ff}-A_{\rm ff})D_{\rm ff}^{-1},\\
S^{T}A(I-\Pi_{A})S&=A_{\rm ff}-\Delta_{\rm ff}D_{\rm ff}^{-1}A_{\rm fc}(P^{T}AP)^{-1}A_{\rm cf}D_{\rm ff}^{-1}\Delta_{\rm ff}.
\end{align*}
The positive semidefiniteness of $\begin{pmatrix}
D_{\rm ff} & A_{\rm fc} \\
A_{\rm cf} & A_{\rm cc}
\end{pmatrix}$ implies that of the Schur complement $A_{\rm cc}-A_{\rm cf}D_{\rm ff}^{-1}A_{\rm fc}$. Then
\begin{align*}
&\sup_{\mathbf{v}_{\rm f}\in\mathbb{R}^{n_{\rm f}}\backslash\mathcal{N}(\Delta_{\rm ff})}\frac{\mathbf{v}_{\rm f}^{T}\Delta_{\rm ff}D_{\rm ff}^{-1}A_{\rm fc}(P^{T}AP)^{-1}A_{\rm cf}D_{\rm ff}^{-1}\Delta_{\rm ff}\mathbf{v}_{\rm f}}{\mathbf{v}_{\rm f}^{T}\Delta_{\rm ff}\mathbf{v}_{\rm f}}\\
&=\sup_{\mathbf{w}_{\rm f}\in\mathcal{R}(\Delta_{\rm ff})\backslash\{0\}}\frac{\mathbf{w}_{\rm f}^{T}\Delta_{\rm ff}^{\frac{1}{2}}D_{\rm ff}^{-1}A_{\rm fc}(P^{T}AP)^{-1}A_{\rm cf}D_{\rm ff}^{-1}\Delta_{\rm ff}^{\frac{1}{2}}\mathbf{w}_{\rm f}}{\mathbf{w}_{\rm f}^{T}\mathbf{w}_{\rm f}}\\
&\leq\lambda_{\max}\big(\Delta_{\rm ff}^{\frac{1}{2}}D_{\rm ff}^{-1}A_{\rm fc}(P^{T}AP)^{-1}A_{\rm cf}D_{\rm ff}^{-1}\Delta_{\rm ff}^{\frac{1}{2}}\big)\\
&=\lambda_{\max}\big(A_{\rm cf}D_{\rm ff}^{-1}\Delta_{\rm ff}D_{\rm ff}^{-1}A_{\rm fc}\big(A_{\rm cc}-A_{\rm cf}D_{\rm ff}^{-1}A_{\rm fc}+A_{\rm cf}D_{\rm ff}^{-1}\Delta_{\rm ff}D_{\rm ff}^{-1}A_{\rm fc}\big)^{-1}\big)\\
&\leq 1,
\end{align*}
which leads to
\begin{displaymath}
\Delta_{\rm ff}\succeq\Delta_{\rm ff}D_{\rm ff}^{-1}A_{\rm fc}(P^{T}AP)^{-1}A_{\rm cf}D_{\rm ff}^{-1}\Delta_{\rm ff}
\end{displaymath}
and hence 
\begin{equation}\label{SApi-low2}
S^{T}A(I-\Pi_{A})S\succeq A_{\rm ff}-\Delta_{\rm ff}=D_{\rm ff}.
\end{equation}

Using~\eqref{new-id} and the positive semidefiniteness of $S^{T}A(I-\Pi_{A})S-D_{\rm ff}$, we obtain
\begin{align*}
\|E_{\rm TL}\|_{A}&=1-\lambda_{\min}\big(\widetilde{M}_{\rm s}^{-1}S^{T}A(I-\Pi_{A})S\big)\\
&\leq 1-\lambda_{\min}\big(\widetilde{M}_{\rm s}^{-1}D_{\rm ff}\big)\\
&=1-\omega^{-2}\lambda_{\min}\big(2\omega I_{n_{\rm f}}-D_{\rm ff}^{-1}A_{\rm ff}\big)\\
&=1-\omega^{-2}\big(2\omega-\lambda_{\max}\big(D_{\rm ff}^{-1}A_{\rm ff}\big)\big),
\end{align*}
which, together with the fact $(1+\varepsilon)D_{\rm ff}\succeq A_{\rm ff}$, yields~\eqref{new-TLr1}.

(ii) Let $N_{\rm s}\in\mathbb{R}^{n_{\rm s}\times n_{\rm s}}$ be an equivalent smoother defined by the relation
\begin{displaymath}
I-SN_{\rm s}^{-1}S^{T}A=\big(I-SM_{\rm s}^{-1}S^{T}A\big)^{\nu},
\end{displaymath}
and let
\begin{displaymath}
R_{\rm ff}=I_{n_{\rm f}}-\omega^{-1}D_{\rm ff}^{-1}A_{\rm ff}.
\end{displaymath}
Then
\begin{align*}
SN_{\rm s}^{-1}S^{T}A&=I-\big(I-SM_{\rm s}^{-1}S^{T}A\big)^{\nu}\\
&=I-\big(I-\omega^{-1}SD_{\rm ff}^{-1}S^{T}A\big)^{\nu}\\
&=\omega^{-1}SD_{\rm ff}^{-1}S^{T}A\sum_{k=0}^{\nu-1}\big(I-\omega^{-1}SD_{\rm ff}^{-1}S^{T}A\big)^{k}\\
&=\omega^{-1}S\Bigg(\sum_{k=0}^{\nu-1}R_{\rm ff}^{k}\Bigg)D_{\rm ff}^{-1}S^{T}A,
\end{align*}
which, together with the fact that $S$ is of full column rank, gives
\begin{displaymath}
N_{\rm s}^{-1}=\omega^{-1}\Bigg(\sum_{k=0}^{\nu-1}R_{\rm ff}^{k}\Bigg)D_{\rm ff}^{-1}=\big(I_{n_{\rm f}}-R_{\rm ff}^{\nu}\big)A_{\rm ff}^{-1}.
\end{displaymath}
Due to
\begin{displaymath}
\lambda(R_{\rm ff})\subset(-1,1) \quad \text{and} \quad R_{\rm ff}A_{\rm ff}^{-1}=A_{\rm ff}^{-1}R_{\rm ff}^{T},
\end{displaymath}
it follows that the inverse $N_{\rm s}^{-1}$ is well defined and symmetric (so is $N_{\rm s}$). Note that $A_{\rm ff}^{\frac{1}{2}}R_{\rm ff}A_{\rm ff}^{-\frac{1}{2}}$ is symmetric and
\begin{displaymath}
\lambda\big(A_{\rm ff}^{\frac{1}{2}}R_{\rm ff}A_{\rm ff}^{-\frac{1}{2}}\big)\subset(-1,1).
\end{displaymath}
We then have
\begin{align*}
N_{\rm s}+N_{\rm s}^{T}-A_{\rm s}&=2A_{\rm ff}\big(I_{n_{\rm f}}-R_{\rm ff}^{\nu}\big)^{-1}-A_{\rm ff}\\
&=2A_{\rm ff}^{\frac{1}{2}}\Big(A_{\rm ff}^{\frac{1}{2}}\big(I_{n_{\rm f}}-R_{\rm ff}^{\nu}\big)^{-1}A_{\rm ff}^{-\frac{1}{2}}\Big)A_{\rm ff}^{\frac{1}{2}}-A_{\rm ff}\\
&=2A_{\rm ff}^{\frac{1}{2}}\Big(I_{n_{\rm f}}-\big(A_{\rm ff}^{\frac{1}{2}}R_{\rm ff}A_{\rm ff}^{-\frac{1}{2}}\big)^{\nu}\Big)^{-1}A_{\rm ff}^{\frac{1}{2}}-A_{\rm ff}\\
&\succ 0.
\end{align*}
Define
\begin{displaymath}
\widetilde{N}_{\rm s}:=N_{\rm s}^{T}\big(N_{\rm s}+N_{\rm s}^{T}-A_{\rm s}\big)^{-1}N_{\rm s} \quad \text{and} \quad \varphi(x):=\frac{1-\big(1-\omega^{-1}x\big)^{2\nu}}{x}.
\end{displaymath}
Then
\begin{align*}
\|E_{\rm TL}\|_{A}&=1-\lambda_{\min}\big(\widetilde{N}_{\rm s}^{-1}S^{T}A(I-\Pi_{A})S\big)\\
&\leq 1-\lambda_{\min}\big(\widetilde{N}_{\rm s}^{-1}D_{\rm ff}\big)\\
&=1-\lambda_{\min}\big(\big(I_{n_{\rm f}}-R_{\rm ff}^{2\nu}\big)A_{\rm ff}^{-1}D_{\rm ff}\big)\\
&=1-\lambda_{\min}\big(\big[I_{n_{\rm f}}-\big(I_{n_{\rm f}}-\omega^{-1}D_{\rm ff}^{-1}A_{\rm ff}\big)^{2\nu}\big]A_{\rm ff}^{-1}D_{\rm ff}\big)\\
&\leq 1-\min_{x\in[1,1+\varepsilon]}\varphi(x),
\end{align*}
where we have used the fact $\lambda\big(D_{\rm ff}^{-1}A_{\rm ff}\big)\subset[1,1+\varepsilon]$.

Straightforward calculations yield
\begin{displaymath}
\frac{\mathrm{d}\varphi(x)}{\mathrm{d}x}=\frac{\psi(x)}{x^{2}} \quad \text{and} \quad \frac{\mathrm{d}\psi(x)}{\mathrm{d}x}=-2\omega^{-2}\nu(2\nu-1)x\big(1-\omega^{-1}x\big)^{2\nu-2},
\end{displaymath}
where
\begin{displaymath} \psi(x)=\big(1+(2\nu-1)\omega^{-1}x\big)\big(1-\omega^{-1}x\big)^{2\nu-1}-1.
\end{displaymath}
It is easy to see that $\psi(x)$ is a decreasing function on $[1,1+\varepsilon]$, which leads to
\begin{displaymath}
\psi(x)\leq\psi(1)=\big(1+(2\nu-1)\omega^{-1}\big)\big(1-\omega^{-1}\big)^{2\nu-1}-1.
\end{displaymath}
\begin{itemize}

\item If $\omega\geq 1$, then
\begin{displaymath}
\psi(1)\leq\big(1-\omega^{-2}\big)^{2\nu-1}-1<0,
\end{displaymath}
where we have used the Bernoulli inequality
\begin{displaymath}
1+(2\nu-1)\omega^{-1}\leq\big(1+\omega^{-1}\big)^{2\nu-1}.
\end{displaymath}

\item If $\frac{1}{2}(1+\varepsilon)<\omega<1$, then $1-\omega^{-1}<0$, which implies that $\psi(1)<0$.

\end{itemize}
Hence, it always holds that $\psi(x)<0$ on the interval $[1,1+\varepsilon]$, i.e., $\varphi(x)$ is a strictly decreasing function on $[1,1+\varepsilon]$. Thus,
\begin{displaymath}
\|E_{\rm TL}\|_{A}\leq 1-\min_{x\in[1,1+\varepsilon]}\varphi(x)=1-\varphi(1+\varepsilon),
\end{displaymath}
which gives the estimate~\eqref{new-TLr2}.

(iii) Obviously, the estimate~\eqref{new-TLr2} will reduce to~\eqref{new-TLr1} if $\nu=1$. Since
\begin{displaymath}
\frac{\mathrm{d}\big(1-\omega^{-1}(1+\varepsilon)\big)^{2\nu}}{\mathrm{d}\omega}=2\nu\big(1-\omega^{-1}(1+\varepsilon)\big)^{2\nu-1}\omega^{-2}(1+\varepsilon),
\end{displaymath}
we conclude that the upper bound in~\eqref{new-TLr2} first strictly decreases and then strictly increases with respect to $\omega$. Moreover, its minimum, $\frac{\varepsilon}{1+\varepsilon}$, is attained if and only if $\omega=1+\varepsilon$. This completes the proof.
\end{proof}

\begin{remark}\label{rmk:comp}
The optimal bound $\frac{\varepsilon}{1+\varepsilon}$ stated in Theorem~\ref{thm:TLr} is less than the upper bounds in~\eqref{old-TLr1} and~\eqref{old-TLr2}. In addition, the estimate~\eqref{old-TLr1} (resp.,~\eqref{old-TLr2}) can be obtained from~\eqref{new-TLr1} (resp.,~\eqref{new-TLr2}) by taking $\omega=1+\frac{\varepsilon}{2}$.
\end{remark}

\begin{remark}
A more general reduction-based two-level method was analyzed in~\cite{Gossler2016}: on the one hand, two independent approximations to $A_{\rm ff}$ were respectively used to define the smoother and prolongation matrix (see also~\cite{Brannick2010}); on the other hand, some polynomials (including the Chebyshev ones) were used for designing the smoothing process. We remark that the upper bounds for $\|E_{\rm TL}\|_{A}$ proved in~\cite[Theorems~3.3 and~4.1]{Gossler2016} are \textit{strictly greater} than $1-\alpha$ ($\alpha=\frac{1}{1+\varepsilon}$ in our setting), that is, the optimal bound $\frac{\varepsilon}{1+\varepsilon}$ in Theorem~\ref{thm:TLr} is less than the convergence bounds in~\cite{Gossler2016}.
\end{remark}

\subsection{Further discussions}

In this subsection, we deeply discuss some important topics on the reduction-based two-level method described in Theorem~\ref{thm:TLr}, including both theoretical and numerical aspects.

\subsubsection{On the assumptions of Theorem~{\rm\ref{thm:TLr}}}

Recall that the two main assumptions of Theorem~\ref{thm:TLr} are
\begin{equation}\label{assump-1}
\lambda\big(D_{\rm ff}^{-1}A_{\rm ff}\big)\subset[1,1+\varepsilon]
\end{equation}
and
\begin{equation}\label{assump-2}
\begin{pmatrix}
D_{\rm ff} & A_{\rm fc} \\
A_{\rm cf} & A_{\rm cc}
\end{pmatrix}\succeq 0.
\end{equation}
The assumption~\eqref{assump-1} can be equivalently expressed as
\begin{displaymath}
\mathbf{v}_{\rm f}^{T}D_{\rm ff}\mathbf{v}_{\rm f}\leq\mathbf{v}_{\rm f}^{T}A_{\rm ff}\mathbf{v}_{\rm f}\leq(1+\varepsilon)\mathbf{v}_{\rm f}^{T}D_{\rm ff}\mathbf{v}_{\rm f} \quad \forall\,\mathbf{v}_{\rm f}\in\mathbb{R}^{n_{\rm f}}.
\end{displaymath}
Theorem~\ref{thm:TLr} suggests that the convergence factor $\|E_{\rm TL}\|_{A}$ can be bounded by a constant that depends only on the spectral equivalence parameter between $D_{\rm ff}$ and $A_{\rm ff}$, indicating uniform convergence in the $A$-norm.

Of particular interest is the design of an easy-to-invert approximation $D_{\rm ff}$ to $A_{\rm ff}$ satisfying~\eqref{assump-1} and~\eqref{assump-2}. Such an approximation can be easily constructed when the SPD matrix $A$ is diagonally dominant, as discussed below.

Let $A=(a_{ij})_{n\times n}$ be partitioned as in~\eqref{two-by-two}. Define
\begin{displaymath}
\theta_{i}:=\frac{a_{ii}}{\displaystyle\sum_{j=1}^{n_{\rm f}}|a_{ij}|} \quad (i=1,\ldots,n_{\rm f}) \quad \text{and} \quad \theta_{\min}:=\min\{\theta_{1},\ldots,\theta_{n_{\rm f}}\}.
\end{displaymath}
The definition of $\theta_{i}$ implies that $\theta_{i}\in(0,1]$ and hence $\theta_{\min}\in(0,1]$. Note that $\theta_{\min}$ is strictly less than $1$ in general. Otherwise, $A_{\rm ff}$ must be a diagonal matrix, in which case one can directly take $D_{\rm ff}=A_{\rm ff}$. In addition, if $A$ is diagonally dominant, then $A_{\rm ff}$ is diagonally dominant as well and hence
\begin{displaymath}
\theta_{i}=\frac{a_{ii}}{a_{ii}+\displaystyle\sum_{\substack{j=1 \\ j\neq i}}^{n_{\rm f}}|a_{ij}|}\geq\frac{1}{2}.
\end{displaymath}
It was proved in~\cite[Corollary~2 and Theorem~5]{MacLachlan2007} that, if $A$ is diagonally dominant and $\theta_{i}\in\big(\frac{1}{2},1\big]$ for all $i=1,\ldots,n_{\rm f}$, then
\begin{equation}\label{Dff}
D_{\rm ff}=\diag\big(\big(2-\theta_{1}^{-1}\big)a_{11},\ldots,\big(2-\theta_{n_{\rm f}}^{-1}\big)a_{n_{\rm f}n_{\rm f}}\big)
\end{equation}
satisfies the assumptions~\eqref{assump-1} and~\eqref{assump-2}, with $\varepsilon=\frac{2(1-\theta_{\min})}{2\theta_{\min}-1}$.

\subsubsection{Lower bounds for $\|E_{\rm TL}\|_{A}$}

Besides upper bounds, we can derive lower bounds for the convergence factor $\|E_{\rm TL}\|_{A}$, which may provide necessary conditions for fast convergence. An obvious fact is
\begin{displaymath}
A_{\rm ff}=S^{T}AS\succeq S^{T}A(I-\Pi_{A})S.
\end{displaymath}
Then
\begin{align*}
\|E_{\rm TL}\|_{A}&=1-\lambda_{\min}\big(\widetilde{M}_{\rm s}^{-1}S^{T}A(I-\Pi_{A})S\big)\\
&\geq 1-\lambda_{\min}\big(\widetilde{M}_{\rm s}^{-1}A_{\rm ff}\big)\\
&=1-\omega^{-2}\lambda_{\min}\big(\big(2\omega I_{n_{\rm f}}-D_{\rm ff}^{-1}A_{\rm ff}\big)D_{\rm ff}^{-1}A_{\rm ff}\big).
\end{align*}
Since
\begin{displaymath}
\lambda\big(D_{\rm ff}^{-1}A_{\rm ff}\big)\subset[1,1+\varepsilon] \quad \text{and} \quad \omega>\frac{1}{2}(1+\varepsilon),
\end{displaymath}
it follows that
\begin{equation}\label{ETG-low1}
\|E_{\rm TL}\|_{A}\geq 1-\omega^{-2}\min\{2\omega-1,\,(2\omega-1-\varepsilon)(1+\varepsilon)\}.
\end{equation}
In the case of multiple pre- and postsmoothing steps, we have
\begin{align*}
\|E_{\rm TL}\|_{A}&=1-\lambda_{\min}\big(\widetilde{N}_{\rm s}^{-1}S^{T}A(I-\Pi_{A})S\big)\\
&\geq 1-\lambda_{\min}\big(\widetilde{N}_{\rm s}^{-1}A_{\rm ff}\big)\\
&=1-\lambda_{\min}\big(I_{n_{\rm f}}-R_{\rm ff}^{2\nu}\big)\\
&=\lambda_{\max}\big(\big(I_{n_{\rm f}}-\omega^{-1}D_{\rm ff}^{-1}A_{\rm ff}\big)^{2\nu}\big),
\end{align*}
which, combined with the fact $\lambda\big(D_{\rm ff}^{-1}A_{\rm ff}\big)\subset[1,1+\varepsilon]$, leads to
\begin{equation}\label{ETG-low2}
\|E_{\rm TL}\|_{A}\geq\max\big\{(1-\omega^{-1})^{2\nu},\,\big(1-\omega^{-1}(1+\varepsilon)\big)^{2\nu}\big\}.
\end{equation}
It is easy to see that~\eqref{ETG-low1} is a special case of~\eqref{ETG-low2}.

\subsubsection{Upper bounds involving the C.B.S. constant $\gamma$}

A key ingredient in the proof of Theorem~\ref{thm:TLr} is the relation~\eqref{SApi-low2}. Similarly, using~\eqref{SApi-low1}, one can get an estimate for $\|E_{\rm TL}\|_{A}$, given by~\eqref{ETG-up}. Under the settings of Theorem~\ref{thm:TLr}, \eqref{ETG-up} reads
\begin{displaymath}
\|E_{\rm TL}\|_{A}\leq 1-(1-\gamma^{2})\omega^{-2}\lambda_{\min}\big(\big(2\omega I_{n_{\rm f}}-D_{\rm ff}^{-1}A_{\rm ff}\big)D_{\rm ff}^{-1}A_{\rm ff}\big),
\end{displaymath}
which, together with the fact $\lambda\big(D_{\rm ff}^{-1}A_{\rm ff}\big)\subset[1,1+\varepsilon]$, yields
\begin{equation}\label{TLr-gamma}
\|E_{\rm TL}\|_{A}\leq 1-(1-\gamma^{2})\omega^{-2}\min\{2\omega-1,\,(2\omega-1-\varepsilon)(1+\varepsilon)\}.
\end{equation}
Besides $\varepsilon$ and $\omega$, the upper bound in~\eqref{TLr-gamma} depends on the C.B.S. constant $\gamma$. In practice, it is often difficult to compute or assess $\gamma$. As a result, the upper bounds in~\eqref{new-TLr1} and~\eqref{TLr-gamma} are theoretically incomparable in general.

According to~\eqref{TLr-gamma}, one can derive a more tractable upper bound for $\|E_{\rm TL}\|_{A}$ by bounding $1-\gamma^{2}$ from below. By~\eqref{CBS} and~\eqref{SApi-low2}, we have
\begin{align*}
1-\gamma^{2}&=1-\lambda_{\max}\big(A_{\rm s}^{-1}S^{T}APA_{\rm c}^{-1}P^{T}AS\big)\\
&=\lambda_{\min}\big(I_{n_{\rm f}}-A_{\rm s}^{-1}S^{T}APA_{\rm c}^{-1}P^{T}AS\big)\\
&=\lambda_{\min}\big(A_{\rm s}^{-1}S^{T}A(I-\Pi_{A})S\big)\\
&\geq\lambda_{\min}\big(A_{\rm ff}^{-1}D_{\rm ff}\big)\\
&\geq\frac{1}{1+\varepsilon},
\end{align*}
which, combined with~\eqref{TLr-gamma}, gives
\begin{equation}\label{ETG-up1}
\|E_{\rm TL}\|_{A}\leq 1-\omega^{-2}\min\bigg\{\frac{2\omega-1}{1+\varepsilon},\,2\omega-1-\varepsilon\bigg\}.
\end{equation}
If the pre- and postsmoothing steps in Algorithm~\ref{alg:TL} are performed $\nu$ times iteratively, then
\begin{align*}
\|E_{\rm TL}\|_{A}&\leq 1-(1-\gamma^{2})\lambda_{\min}\big(\widetilde{N}_{\rm s}^{-1}A_{\rm s}\big)\\
&=1-(1-\gamma^{2})\lambda_{\min}\big(I_{n_{\rm f}}-R_{\rm ff}^{2\nu}\big)\\
&=1-(1-\gamma^{2})\big(1-\lambda_{\max}\big(\big(I_{n_{\rm f}}-\omega^{-1}D_{\rm ff}^{-1}A_{\rm ff}\big)^{2\nu}\big)\big),
\end{align*}
which leads to
\begin{equation}\label{ETG-up2}
\|E_{\rm TL}\|_{A}\leq 1-\frac{1-\max\big\{(1-\omega^{-1})^{2\nu},\,\big(1-\omega^{-1}(1+\varepsilon)\big)^{2\nu}\big\}}{1+\varepsilon}.
\end{equation}
It is easy to see that~\eqref{ETG-up2} will reduce to~\eqref{ETG-up1} if $\nu=1$. We point out that
\begin{itemize}

\item \eqref{new-TLr1} and~\eqref{new-TLr2} are sharper than~\eqref{ETG-up1} and~\eqref{ETG-up2}, respectively;

\item the upper bounds in~\eqref{ETG-up1} and~\eqref{ETG-up2} can be minimized by $\omega=1+\frac{\varepsilon}{2}$, in which case~\eqref{ETG-up1} and~\eqref{ETG-up2} will reduce to~\eqref{old-TLr1} and~\eqref{old-TLr2}, respectively.

\end{itemize}

\subsubsection{Numerical comparisons}

As indicated in Remark~\ref{rmk:comp}, our new theory shows that $\omega=1+\varepsilon$ can yield a smaller upper bound for $\|E_{\rm TL}\|_{A}$ compared to the weight used in Theorem~\ref{thm:TLr0}. In some cases, $\omega=1+\varepsilon$ may result in a faster algorithm. In what follows, we compare these two weights via a numerical example.

The 2D Poisson's equation with homogeneous Dirichlet boundary conditions can be expressed as
\begin{equation}\label{Poisson}
\left\{
\begin{aligned}
-\Delta u&=f \quad \text{in $\Omega$},\\
u&=0 \quad \text{on $\partial\Omega$},
\end{aligned}
\right.
\end{equation}
where $\Omega=(0,1)\times(0,1)$. Let $\Omega$ be partitioned uniformly in both $x$- and $y$-directions into $(m+1)^{2}$ pieces (i.e., $(m+1)\times(m+1)$ grid), and let
\begin{displaymath}
(x_{i},y_{j})=(ih,jh) \quad \forall\,i,j=0,1,\ldots,m+1,
\end{displaymath}
where $h=\frac{1}{m+1}$. Consider a five-point finite difference stencil for the operator $-\Delta$, given by
\begin{displaymath}
\frac{1}{h^{2}}\begin{bmatrix}
   & -1 &  \\
-1 & 4 & -1 \\
   & -1 & 
\end{bmatrix}.
\end{displaymath}
In our experiments, all functions or their approximations evaluated at $(x_{i},y_{j})$ are ordered lexicographically. Then, the five-point finite difference discretization of~\eqref{Poisson} leads to the linear system
\begin{equation}\label{system-0}
A_{0}\mathbf{u}_{0}=\mathbf{f}_{0},
\end{equation}
where
\begin{displaymath}
A_{0}=I_{m}\otimes B+C\otimes I_{m}
\end{displaymath}
with $B={\rm tridiag}(-1,4,-1)\in\mathbb{R}^{m\times m}$ and $C={\rm tridiag}(-1,0,-1)\in\mathbb{R}^{m\times m}$, and $\mathbf{f}_{0}$ is formed by ordering $\big\{h^{2}f(x_{i},y_{j})\big\}_{i,j=1}^{m}$ lexicographically.

According to the strategy of designing $D_{\rm ff}$ (see~\eqref{Dff}), the block $A_{\rm ff}$ is expected to be sufficiently diagonally dominant. To get such a subblock, we apply the \textit{greedy coarsening algorithm} in~\cite[Algorithm~3]{MacLachlan2007} to $A_{0}$, with threshold $\theta=0.55$. Based on the resulting fine and coarse points, we reorder the entries and unknowns in~\eqref{system-0} accordingly, that is, the original system~\eqref{system-0} will be equivalently transformed into
\begin{equation}\label{system-new}
\begin{pmatrix}
\mathscr{P}_{\rm f}^{T} \\
\mathscr{P}_{\rm c}^{T}
\end{pmatrix}A_{0}(\mathscr{P}_{\rm f} \,\ \mathscr{P}_{\rm c})\begin{pmatrix}
\mathscr{P}_{\rm f}^{T} \\
\mathscr{P}_{\rm c}^{T}
\end{pmatrix}\mathbf{u}_{0}=\begin{pmatrix}
\mathscr{P}_{\rm f}^{T} \\
\mathscr{P}_{\rm c}^{T}
\end{pmatrix}\mathbf{f}_{0},
\end{equation}
where $(\mathscr{P}_{\rm f} \,\ \mathscr{P}_{\rm c})\in\mathbb{R}^{m^{2}\times m^{2}}$ is a permutation matrix such that $\mathscr{P}_{\rm f}^{T}\mathbf{u}_{0}$ and $\mathscr{P}_{\rm c}^{T}\mathbf{u}_{0}$ are low-dimensional vectors consisting of fine- and coarse-grid unknowns, respectively.

After the permutation process, we apply Algorithm~\ref{alg:TL} to the linear system~\eqref{system-new}. Main settings in our experiments are listed below:
\begin{itemize}

\item $f=2\pi^{2}\sin(\pi x)\sin(\pi y)$;

\item $\mathbf{u}^{(0)}$ is chosen randomly and then fixed;

\item $S=\begin{pmatrix}
I_{n_{\rm f}} \\
0
\end{pmatrix}$;

\item $M_{\rm s}=\omega D_{\rm ff}$, where $D_{\rm ff}$ is given by~\eqref{Dff};

\item $P=\begin{pmatrix}
-D_{\rm ff}^{-1}A_{\rm fc} \\
I_{n_{\rm c}}
\end{pmatrix}$ with $A_{\rm fc}=\mathscr{P}_{\rm f}^{T}A_{0}\mathscr{P}_{\rm c}$.

\end{itemize}
The numerical results are shown in Table~\ref{tab:rhos}, where $\varepsilon$ is chosen as $\lambda_{\max}\big(D_{\rm ff}^{-1}A_{\rm ff}\big)-1$, $\rho_{1}$ and $\rho_{2}$ denote the asymptotic convergence factors of Algorithm~\ref{alg:TL} with weights $\omega_{1}=1+\varepsilon$ and $\omega_{2}=1+\frac{\varepsilon}{2}$, respectively.

\vskip -0.2cm

\begin{table}[!htbp]
\centering
\captionsetup{width=\linewidth}
\begin{tabular}{p{2cm}<{\centering}|p{1cm}<{\centering}p{1cm}<{\centering}p{1cm}<{\centering}p{1cm}<{\centering}p{1.2cm}<{\centering}p{1.2cm}<{\centering}}
\bottomrule \\ [-2.5ex]
Grid & $\varepsilon$ & $\omega_{1}$ & $\omega_{2}$ & $\nu$ & $\rho_{1}$ & $\rho_{2}$ \\
\hline \\ [-2.5ex]
&  &  &  & $1$ & $0.6269$ & $0.7247$  \\[0.25ex]
$32\times 32$ & $3.8022$ & $4.8022$ & $2.9011$ & $2$ & $0.5426$ & $0.6101$  \\[0.25ex]
&  &  &  & $3$ & $0.5334$ & $0.5626$  \\
\hline \\ [-2.5ex]
&  &  &  & $1$ & $0.6272$ & $0.7274$  \\[0.25ex]
$64\times 64$ & $3.8067$ & $4.8067$ & $2.9034$ & $2$ & $0.5459$ & $0.6135$  \\[0.25ex]
&  &  &  & $3$ & $0.5369$ & $0.5662$  \\
\hline \\ [-2.5ex]
&  &  &  & $1$ & $0.6273$ & $0.7285$  \\[0.25ex]
$128\times 128$ & $3.8078$ & $4.8078$ & $2.9039$ & $2$ & $0.5473$ & $0.6150$  \\[0.25ex]
&  &  &  & $3$ & $0.5383$ & $0.5677$  \\
\toprule
\end{tabular}
\vskip -0.2cm
\caption{\centering\small Numerical comparisons of the asymptotic convergence factors $\rho_{1}$ and $\rho_{2}$.}
\label{tab:rhos}
\end{table}

\vskip -0.2cm

From the last two columns of Table~\ref{tab:rhos}, we can observe that the weight $\omega_{1}=1+\varepsilon$ yields smaller convergence factors compared to $\omega_{2}=1+\frac{\varepsilon}{2}$. In addition, the resulting convergence factors do not degrade with respect to grid size.

\section{Conclusions} \label{sec:con}

The two-level convergence identity plays an important role in the analysis and design of multilevel methods, because many multilevel methods can be viewed as a perturbed variant of Algorithm~\ref{alg:TL}. In this paper, we present an easy-to-use identity for characterizing the convergence factor of two-level methods, whose hierarchical spaces can be either overlapping or non-overlapping. Furthermore, several applications have been provided to illustrate its usability and convenience.

\section*{Acknowledgments}

This work was partially supported by the National Natural Science Foundation of China (Grant No.\,12401479), the Natural Science Foundation of Jiangsu Province (Grant No.\,BK20241257), and the Start-up Research Fund of Southeast University (Grant No.\,RF1028623372).

\bibliographystyle{amsplain}

\end{document}